\documentclass[11pt, dvipsnames, oneside]{amsart}
\usepackage{geometry}

\usepackage{amsmath, amssymb, amsthm}
\usepackage{lastpage}

\usepackage[colorlinks=true,urlcolor=blue]{hyperref}

\usepackage{bm}

\usepackage{pgf,tikz}
\usetikzlibrary{arrows}

\usepackage{xcolor}
\usepackage{graphicx}

\newtheorem{MainTheorem}{Theorem}

\newtheorem{MainCorollary}{Corollary}

\newtheorem{theorem}{Theorem}[section]

\newtheorem{proposition}{Proposition}[section]

\newtheorem*{bl2d}{Blaschke's Inclusion Theorem in $\mathbb R^2$}

\theoremstyle{definition}
\newtheorem{definition}{Definition}[section]

\newtheorem{remark}{Remark}[section]

\renewcommand{\ge}{\geqslant}
\renewcommand{\le}{\leqslant}

\DeclareMathOperator{\sm}{\setminus}

\DeclareMathOperator{\dist}{dist}
\DeclareMathOperator{\di}{d}
\DeclareMathOperator{\grad}{grad}

\DeclareMathOperator{\sn}{sn}
\DeclareMathOperator{\ct}{ct}

\DeclareMathOperator{\inj}{inj}
\DeclareMathOperator{\spa}{span}
\DeclareMathOperator{\Vol}{Vol}

\DeclareMathOperator{\spn}{span}
\DeclareMathOperator{\diam}{diam}
\DeclareMathOperator{\conv}{conv}

\DeclareMathOperator{\Focal}{Focal}
\DeclareMathOperator{\Roll}{Roll}

\numberwithin{equation}{section}
\usepackage{caption}
\DeclareCaptionLabelSeparator{dot}{. }
\title[The Blaschke rolling theorem in Riemannian manifolds]{The Blaschke rolling theorem in Riemannian manifolds of bounded curvature}
\author[Kostiantyn Drach]{Kostiantyn Drach}
\address{Universitat de Barcelona, Departament de Matem\`atiques i Inform\`atica, Gran Via 585, 08007 Barcelona, Spain}
\address{Centre de Recerca Mathem\`atica, Edifici C, Carrer de l'Albareda, 08193 Bellaterra, Spain}
\email{kostiantyn.drach@ub.edu}

\thanks{\textbf{Acknowledgements.}
The author is grateful to Alexander Borisenko for many useful and illuminating discussions on the topic of Blaschke's theorem throughout the years, and to Sebastian Boldt for continued encouragement to finish this project. The author also acknowledges partial support from the Departament de Recerca i Universitats de la Generalitat de Catalunya (2021 SGR 00697).
}

\begin{document}

\maketitle

\begin{abstract}
We generalize the classical Blaschke Rolling Theorem to convex domains in Riemannian manifolds of bounded sectional curvature and arbitrary dimension. Our results are sharp and, in this sharp form, are new even in the model spaces of constant curvature.
\end{abstract}

%{
%\hypersetup{linktoc=page}
%\setcounter{tocdepth}{1}
%\tableofcontents
%}

%%%%%%%%%%%%%%%%%%%%%%%%%%%%%%%%%%%%%%%%%%%%%%%%%%%%%%
%%%%%%%%%%%%%%%%%%%%%%%%%%%%%%%%%%%%%%%%%%%%%%%%%%%%%%

\section{Introduction}

%%%%%%%%%%%%%%%%%%%%%%%%%%%%%%%%%%%%%%%%%%%%%%%%%%%%%%
%%%%%%%%%%%%%%%%%%%%%%%%%%%%%%%%%%%%%%%%%%%%%%%%%%%%%%

The Blaschke inclusion theorem, a classical result from the field of global differential geometry, was proven by Wilhelm Blaschke in his famous monograph (see \cite[pp. 114-117]{Bla56}) more than a century ago\footnote{The 1st edition of the book was published in 1916.}. He showed the following (see Figure~\ref{Fig:Bla1}):

\begin{bl2d}
If $B_1$ and $ B_2$ are two strictly convex, compact planar domains with smooth boundaries touching at some point $p \in \partial B_1 \cap \partial B_2$ in such a way that they share a common inner normal, then 
$$
B_1 \subseteq B_2
$$
provided that the curvatures $k_i (\nu)$ of the boundaries, as functions of inner normals $\nu \in \mathbb S^1$, satisfy 
\begin{equation}
\label{CurvIneq}
k_1 (\nu) \geqslant k_2 (\nu) \text{ for all } \nu \in \mathbb S^1.
\end{equation}
\end{bl2d}

\begin{figure}[h]
\definecolor{ffqqqq}{rgb}{1.,0.,0.}
%\definecolor{ForestGreen}{rgb}{0.,1.,0.}
%\definecolor{NavyBlue}{rgb}{0.,0.,1.}
\begin{tikzpicture}[line cap=round,line join=round,>=stealth,x=1.0cm,y=1.0cm, scale=1.25]
\clip(-1,-1.72) rectangle (7,3.3);
\draw [rotate around={47.20259816176573:(2.56,0.78)},line width=1.pt,color=orange] (2.56,0.78) ellipse (2.27645031990803cm and 1.7366133878930472cm);
\draw [rotate around={49.83041995828995:(2.53,1.63)},line width=1.pt,color=ForestGreen] (2.53,1.63) ellipse (1.3502565240950077cm and 0.8987728750141095cm);
\draw [dash pattern=on 1pt off 2pt,color=NavyBlue,line width=1.pt] (1.3159230951130156,2.565448109957082)-- (3.9,3.0);
\draw (2.4,3.4) node[anchor=north west] {$p$};
\draw (0.82,-1) node[anchor=north west] {$\textcolor{orange}{B_2}$};
\draw (1.2,0.58) node[anchor=north west] {$\textcolor{ForestGreen}{B_1}$};
\draw ((2.9,1.75) node[anchor=west] {$\nu$};
\draw [rotate around={49.83041995828995:(3.355,1.25)},line width=0.5pt,color=ForestGreen] (3.355,1.25) ellipse (1.3502565240950077cm and 0.8987728750141095cm);
\draw [rotate around={49.83041995828995:(3.355,0.5)},line width=0.5pt,color=ForestGreen] (3.355,0.5) ellipse (1.3502565240950077cm and 0.8987728750141095cm);
\draw [->,line width=1.2pt,color=NavyBlue] (2.78,2.8) -- (2.949450129046894,1.742289399491584);
\draw [fill=black] (2.78,2.8) circle (1.5pt);
\end{tikzpicture}
\caption{Blaschke's inclusion theorem.}
\label{Fig:Bla1}
\end{figure}
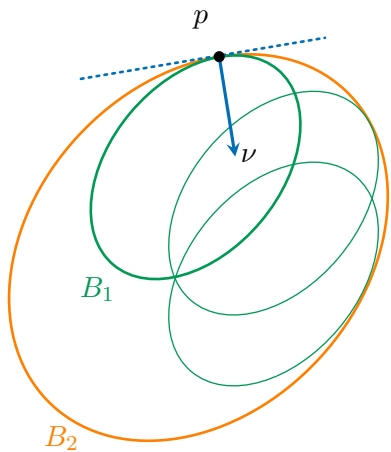

Informally speaking, if the boundary of $B_1$ is ``curved more'' than the boundary of $B_2$, then $B_1$ can ``slide'' along $\partial B_2$ while staying inside $B_2$, as indicated in Figure~\ref{Fig:Bla1}. An interesting special case is when one of the domains is a disk and thus one of the curvatures $k_i$ is constant. Traditionally, this particular case is distinguished from Blaschke's inclusion result and is referred to as \textit{the Blaschke rolling theorem} --- the corresponding disk can either be rolled inside the domain (along the boundary), or, vise versa, the domain can roll inside the disk. 

Blaschke's inclusion (and rolling) theorem was later generalized by many authors. Although extension of the \emph{rolling} result from curves in $\mathbb R^2$ to surfaces in $\mathbb R^n$, appropriately changing~(\ref{CurvIneq}) to an inequality between the second fundamental forms, follows from a section--type (if $B_2$ is a ball) or a projection--type (if $B_1$ is a ball) argument, the general inclusion case requires some technical work. Koutroufiotis~\cite{Kou72} generalized Blaschke's inclusion theorem for complete curves in $\mathbb R^2$ and complete surfaces in $\mathbb R^3$. He also treated the case when $B_1$ and $B_2$ touch in a more than one point. Later, the result was extended to compact surfaces in $\mathbb R^n$ by Rauch~\cite{Rau74}, and to complete surfaces by Delgado~\cite{Del79}. Both of them went along the original ideas of Blaschke. Brooks and Strantzen~\cite{BrStr89}, by carefully analyzing previously known arguments, generalized Blaschke's inclusion theorem for non-smooth strictly convex sets. 

While extensively treated in the Euclidean case, Blaschke's theorem received less attention in general Riemannian setting, where similar inclusion-type questions make perfect sense upon appropriate bounds on curvatures. Karcher~\cite{Kar68} proved the Blaschke rolling theorem for \textit{compact domains with the smooth boundary} in model two-dimensional geometries (Euclidean, hyperbolic and spherical; see also the reference in~\cite{Kar68} for some partial results due to Santal\'o). Moreover, Karcher showed that, unlike the spherical and Euclidean cases, Blaschke's rolling theorem is generally false for convex domains with an upper bound on the curvature of the boundary lying in the hyperbolic space. Therefore, one must impose natural restrictions to achieve a proper generalization (see~\cite[Section 8]{Kar68}). In~\cite{Mil70} Milka, using techniques of global synthetic geometry, extended the results of Karcher by proving Blaschke's rolling theorem in two-dimensional model spaces without restrictions on smoothness of the boundary by using a concept of specific curvature of the curve. Milka's techniques also work for unbounded convex domains. 

For some further progress on the discrete versions of Blaschke's theorem in $\mathbb R^2$ we refer to a paper of R\'ev\'esz~\cite{Re16} and, for convex polygons in two-dimensional model spaces, to the work of Borisenko and Miquel~\cite{BorMiqBla1}. The latter result was partially extended by Borisenko and Miquel to convex polygons in Hadamard manifolds of dimension $2$ \cite{BorMiqBla2}. 

Multidimensional constant curvature case of Blaschke's theorem seems to be folklore in the community due to seemingly straightforward reduction to the two-dimensional case via projection/section-type arguments or via the so-called \textit{dual map} (see, for example, \cite[Theorems 9.2.5, 10.4.4]{Ger06} for the definition and properties of dual maps). The multidimensional Riemannian case of varying curvature is more delicate. In~\cite{How99}, Howard proved a far reaching generalization of Blaschke's rolling theorem for Riemannian manifolds of bounded geometry with boundary. His generalization concerns the case when a ball \textit{rolls inside} a domain and is formulated in terms of equality of rolling radius $\Roll(M)$ and focal radius $\Focal(\partial M)$ of a Riemannian manifold $M$ and its boundary $\partial M$. In particular, \cite[Theorems 1.1, 4.4]{How99} imply the inner Blaschke's rolling theorem for convex domains in constant curvature spaces, using, for example,~\cite[Proposition 1.1]{CR78} to estimate $\Focal(\partial M)$ in terms of principal curvatures). Using completely different techniques, the \textit{outer} Blaschke rolling theorem, that is when a domain rolls inside a ball, was implicitly proven in $n$-dimensional hyperbolic space in~\cite{BM02} by Borisenko and Miquel and in the remaining model spaces of constant curvature, in~\cite{BDr13} by Borisenko and the author. The approach to Blaschke's theorem in~\cite{BDr13} and~\cite{BM02} was based on the so-called \textit{radial angle comparison theorem} -- a powerful tool to study strictly convex hypersurfaces with bounds on their principal curvatures, pioneered by Borisenko, and later extended in the series of publications (e.g., \cite{BGR01, BM02, Bor02, BDr13}). 

In this paper we establish the sharp version of outer Blaschke's rolling theorem for uniformly convex domains in Riemannian manifolds of bounded sectional curvature. As an ingredient, we use a sharp version of the radial angle comparison theorem, for which we provide a simplified (compared to those in~\cite{BM02, BDr13}) proof. In the last section, we will pose some open questions related to Blaschke's rolling theorem in spaces of non-constant curvature.

%%%%%%%%%%%%%%%%%%%%%%%%%%%%%%%%%%%%%%%%%%%%%%%%%%%%%%%
%%%%%%%%%%%%%%%%%%%%%%%%%%%%%%%%%%%%%%%%%%%%%%%%%%%%%%%

\section{Preliminaries and statement of the main results}

%%%%%%%%%%%%%%%%%%%%%%%%%%%%%%%%%%%%%%%%%%%%%%%%%%%%%%%
%%%%%%%%%%%%%%%%%%%%%%%%%%%%%%%%%%%%%%%%%%%%%%%%%%%%%%%

%%%%%%%%%%%%%%%%%%%%%%%%%%%%%%%%%%%%%%%%%%%%%%%%%%%%%%%
\subsection{Preliminaries}
%%%%%%%%%%%%%%%%%%%%%%%%%%%%%%%%%%%%%%%%%%%%%%%%%%%%%%%

Let $(M, \left<\cdot, \cdot\right>)$ be a complete, smooth, orientable Riemannian manifold of dimension $m \geqslant 2$.

We will consistently use the notation $D_\Sigma$ for a closed domain in $M$ with non-empty interior and at least $C^2$-smooth boundary $\Sigma := \partial D_\Sigma$; we assume that $\Sigma$ is connected. Therefore, $\Sigma$ is a closed embedded orientable hypersurface in $M$. Unless stated otherwise, we do not  assume $D_\Sigma$ to be compact. 

Denote by $\nu \colon \Sigma \to N\Sigma$ the unit normal vector field along $\Sigma$ giving a positive orientation to $\Sigma$ with respect to $D_\Sigma$, and for a point $p \in \Sigma$, define $\kappa_{\min}(p)$ to be the smallest principal curvature of $\Sigma$ at $p$ with respect to $\nu(p)$.

\begin{definition}[$\lambda$-convexity]
\label{Def:strictconv}
Given a constant $\lambda>0$, we say that the hypersurface $\Sigma$ and the domain $D_\Sigma$ are \textit{$\lambda$-convex} if 
$$
\kappa_{\min} (p) \geqslant \lambda \quad \text{for all} \quad p \in \Sigma.
$$
\end{definition}

Equivalently, if $B^\nu$ is the second fundamental form of $\Sigma$ with respect to $\nu$, then $D_\Sigma$ is $\lambda$-convex if and only if
$$
B^\nu (v, v) \geqslant \lambda \left<v,v\right> \quad \text{ for all } \quad p \in \Sigma, \quad v \in T_p \Sigma
$$
(here, $T_p \Sigma$ is the tangent space to $\Sigma$ at $p$). The quantity $k^{\nu}(p, v) := {B^{\nu}|_p(v,v)}/{\left<v,v\right>|_p}$ is called the \textit{normal curvature} of $\Sigma$ at the point $p \in \Sigma$ and in the direction of $v \in T_p \Sigma \sm \{0\}$. Therefore, a domain is $\lambda$-convex if and only if all normal curvatures of $\Sigma$ are bounded below by $\lambda$. Note that in the case $m=2$, that is when $\Sigma$ is a closed embedded curve, the normal curvature of $\Sigma$ is just the geodesic curvature of this curve. 

The notion of $\lambda$-convexity, also referred to in the literature as ``hyperconvexity'', ``spindle convexity'', ``ball convexity'', is a generalization of the classical notion of convexity. It naturally appears in various contexts in smooth and discrete geometry, most notably, in the recent advances on the reverse isoperimetric problems (see, e.g., \cite{CDT, DT} and references therein).

From the definition of $\lambda$-convexity it follows that the second fundamental form of $\Sigma$ is positive-definite. If $M$ is $C^2$-smooth and $\Sigma$ is $C^3$-smooth, then by \cite{Convexity}, $D_\Sigma$ is a convex domain in $M$, that is, every two points in $M$ can be joined by a minimizing geodesic within $D_\Sigma$. 

\begin{remark}
The notion of $\lambda$-convexity can be extended to non-smooth hypersurfaces as well; see, e.g., \cite[Definition 2.3]{Bor02} and \cite{BDr13}. Although in this paper we will be focused only on the smooth case, some of the results that we prove can be easily extended to a non-smooth setting. 
\end{remark}

Let $M(c)$ a complete simply-connected $m$-dimensional Riemannian manifold of constant sectional curvature equal $c$; we will refer to such manifolds as to \textit{model spaces of curvature $c$}. As usual, we assume $m \geqslant 2$. By a well-known classification, a model space of curvature $c$ is the spherical space $\mathbb S^{m}(c)$ when $c>0$, the Euclidean space $\mathbb R^{m}$ when $c = 0$, and the hyperbolic space $\mathbb H^{m}(c)$ when $c < 0$. We will suppress the dimension index when it does not cause confusion. 

Recall that a hypersurface $\Sigma$ in a Riemannian manifold is called \textit{totally umbilical (of curvature $\lambda$)} if $\kappa^\nu(p, v) \equiv \lambda$ for all $p \in \Sigma$ and $v \in T_p \Sigma$. In the model space of curvature $c$ one can classify complete totally umbilical hypersurfaces of curvature $\lambda$ (see, e.g., \cite[Section 35.2.4]{BZ}). Up to isometry and scaling, for $c\geqslant 0$ these are \textit{hyperplanes} (that is, complete totally geodesic hypersurfaces) and geodesic \emph{spheres}. However, in the hyperbolic space there are four types of totally umbilical hypersurfaces: hyperplanes (when $\lambda = 0$), \emph{equidistants} (when $1/\sqrt{-c} > \lambda > 0$, also called \emph{hyperspheres}), \emph{horospheres} (when $\lambda = 1/\sqrt{-\lambda}$), and geodesic spheres (when $\lambda > 1/\sqrt{-\lambda}$). Regardless of the value of $c$, spheres are the only compact totally umbilical hypersurfaces in a model space. 

For the later parts, it will be useful to introduce a so-called \textit{generalized sine function}
\begin{equation*}
\label{sin}
\sn_{c}(t) := \left\{
\begin{aligned}
&\frac{1}{\sqrt{c}}\sin \left(t\sqrt{c}\right), \text{ if }c > 0 \\
& t, \text{ if }c = 0 \\
&\frac{1}{\sqrt{-c}}\sinh \left(t\sqrt{-c}\right), \text{ if }c < 0,\\
\end{aligned}
\right.
\end{equation*}
and likewise, \textit{a generalized cotangent function} $\ct_c (t) := {\sn_c'(t)}/{\sn_c(t)}$. In this notation, if $R$ is the radius of a convex geodesic sphere in $M(c)$, then the normal curvature of this sphere is equal to $\ct_c(R)$. In particular, 
\begin{equation}
\label{Eq:Sphere}
\lambda > 0 ,\quad \text{ and moreover } \quad \lambda > \sqrt{-c} \quad \text{ for } \quad c<0.
\end{equation}
We say that $\lambda \in \mathbb R$ \emph{satisfies the convex sphere curvature constraints} (or \emph{sphere constraints} for short) if \eqref{Eq:Sphere} holds. For such $\lambda$, we denote by $R_\lambda$ the \emph{characteristic} radius of the sphere with curvature $\lambda$ in $M(c)$.

Finally, for a subset  $A \subset M$ and a point $x \in M$ the \textit{cone over $A$ with the vertex at $x$}, denoted $\mathcal C(x, A)$, is the union of all geodesic segments connecting $x$ to the points in $A$. 

We use the notation $\inj(M)$ for the global \emph{injectivity radius} of $M$. For a point $x \in M$, the \emph{exponential mapping} $\exp_x \colon T_x M \to M$ is well-defined and is a diffeomorphism of the ball of radius $\inj(M)$ in $T_p M$ onto its image. We will say that a set $S \subset T_x M$ \emph{does not belong to any closed half-space} if for any hyperplane $\Pi$ passing through the origin in $T_xM$, the set $S$ does not lie in any of the closed half-spaces with respect to $\Pi$.  

Furthermore, we denote by $K_\sigma = K_\sigma(x)$ the sectional curvature of $M$ along a $2$-dimensional subspace $\sigma \subset T_x M$. We will write $\sec(M) \ge c$ if $K_\sigma(x) \ge c$ for all $x \in M$ and  $\sigma \subset T_x M$ (and similarly $\sec(M) \le c$).

%%%%%%%%%%%%%%%%%%%%%%%%%%%%%%%%%%%%%%%%%%%%%%%%%%%%%%%%
\subsection{Statements of the main results}
%%%%%%%%%%%%%%%%%%%%%%%%%%%%%%%%%%%%%%%%%%%%%%%%%%%%%%%%

With the given definitions, we are now ready to state our first main result --- a sharp Riemannian version of Blaschke's rolling theorem for balls (see Figure~\ref{Fig:StatementThmA}).

\begin{MainTheorem}[Riemannian Blaschke's ball rolling theorem]
\label{Thm:convexforballs}
Let $M$ be a complete $C^2$-smooth Riemannian manifold of dimension $m \ge 2$ with $\sec(M) \ge c$. Suppose $\lambda \in \mathbb R$ satisfies the sphere constraints with characteristic radius $R_\lambda$, and assume that $\inj(M) > 2R_\lambda$. Then every $\lambda$-convex domain $D_\Sigma \subset M$ with $C^3$-smooth boundary can roll freely inside a geodesic ball of radius $R_\lambda$ in $M$.

\textbf{I.} More precisely, for every point $s \in \Sigma$, if $\gamma \colon [0, 1] \to M$, $\gamma(0) = s$, $\gamma(1) = \bar s$ is the geodesic segment of length $R_\lambda$ such that $\dot \gamma(0)$ is the inward pointing unit normal to $\Sigma$, then the geodesic ball $B(\bar s, R_\lambda)$ of radius $R_\lambda$ centered at $\bar s$ satisfies   
\begin{equation}
\label{Eq:ballincl}
B(\bar s, R_\lambda) \supseteq D_\Sigma.
\end{equation}

\textbf{II.} Moreover, $\Sigma_s := \Sigma \cap B(\bar s, R_\lambda)$ is a (closed) connected subset of $\Sigma$, and exactly one of the following alternatives must hold. 
\begin{enumerate}
\item[(i)]
If $\Sigma_s$ contains $m+1$ distinct points $s_1, \ldots, s_{m+1}$ such that the corresponding $m+1$ vectors $\exp_{\bar s}^{-1}(s_i)$ do not lie in any closed half-space of $T_{\bar s} M$, then $D_\Sigma$ is isometric to a closed ball of radius $R_\lambda$ in the model space of curvature $c$.

\item[(ii)]
Otherwise, the cone $\mathcal C(\bar s, \Sigma_x) \subset M$ over $\Sigma_s$ with the vertex at $\bar s$ is isometric to a cone in the model space of curvature $c$ with the vertex at the center of a sphere of radius $R_\lambda$ over a convex domain in this sphere.  
\end{enumerate}

Both alternatives (i) and (ii) can occur.

\end{MainTheorem}

\begin{figure}[h]
\includegraphics[width=0.5\textwidth, trim = 20 23 20 23, clip]{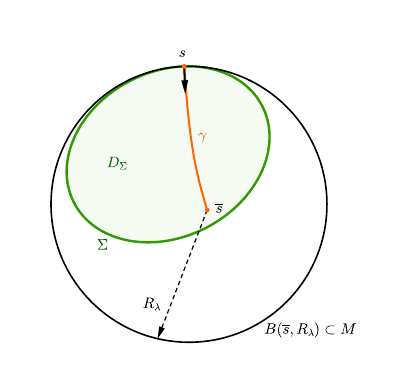}
\caption{Riemannian Blaschke's ball rolling theorem.}
\label{Fig:StatementThmA}
\end{figure}

An immediate corollary to Theorem~\ref{Thm:convexforballs} is the following diameter comparison result for $\lambda$-convex domains.

\begin{MainCorollary}[Diameter comparison for $\lambda$-convex domains]
\label{Cor:Cor1ThmA}
If $D \subset M$ be a $\lambda$-convex domain in a Riemannian manifold $M$ with $\sec(M) \ge c$, both satisfying the assumptions of Theorem~\ref{Thm:convexforballs}, then
\begin{equation*}
\diam(D) \le 2 \ct_c^{-1}(\lambda) = \left\{
\begin{aligned}
&\frac{2}{\sqrt{c}} \cot^{-1}\left(\frac{\lambda}{\sqrt{c}}\right), &\quad \text{if } c > 0,\\
&\frac{2}{\lambda},  &\quad \text{if } c = 0,\\
&\frac{2}{\sqrt{-c}} \coth^{-1}\left(\frac{\lambda}{\sqrt{-c}}\right), &\quad \text{if } c < 0 &\text{ (and $\lambda > \sqrt{-c}$)}.\\
\end{aligned}
\right. 
\end{equation*} 
In particular, when $c > 0$, we have that $\diam(D) \le \pi/\sqrt{c}$.
\qed
\end{MainCorollary}

Another corollary to Theorem~\ref{Thm:convexforballs} is the following sharp volume comparison for $\lambda$-convex domains. It is a generalization to higher dimensions of the result due to Toponogov \cite{Top} (see also Hang--Wang \cite[Theorem 4]{HW}).

\begin{MainCorollary}[Volume comparison for $\lambda$-convex domains]
\label{Cor:Cor2ThmA}
Let $M$ be a Riemannian manifold with $\sec(M) \ge c$ and $D_\Sigma \subset M$ be a $\lambda$-convex domain, both satisfying the assumptions of Theorem~\ref{Thm:convexforballs}. Then
\[
\Vol(D_\Sigma) \le \Vol(B_\lambda), \quad \text{and} \quad \Vol(\Sigma) \le \Vol(\partial B_\lambda),
\]
where $B_\lambda \subset M(c)$ is a ball of radius $R_\lambda$ in the model space of curvature $c$. Moreover, the equality holds if and only if $D_\Sigma$ is isometric to $B_\lambda$.  
\end{MainCorollary}

\begin{proof}
Inclusion~\eqref{Eq:ballincl} applied to $D_\Sigma$ implies that $\Vol(D_\Sigma) \le \Vol(B(\overline s, R_\lambda))$ and $\Vol(\Sigma) \le \Vol(\partial B(\overline s, R_\lambda))$ \cite{Area} for any $\overline s$. By the volume comparison for geodesic balls \cite[Lemma 7.1.3]{Pet}, $\Vol(B(\overline s, R_\lambda)) \le \Vol(B_\lambda)$ and $\Vol(\partial B(\overline s, R_\lambda)) \le \Vol(\partial B_\lambda)$, and the claim follows. The equality case follows by the rigidity part in Theorem~\ref{Thm:convexforballs}. 
\end{proof}

In Theorem~\ref{Thm:convexforballs}, when $c<0$, we assumed that $\lambda > \sqrt{-c}$ in order to ensure that in the model hyperbolic space $\mathbb H^m(c)$ there exists a ball of curvature $\lambda$. However, as it was mentioned above, in the hyperbolic space there exist two more types of hypersurfaces that are totally umbilical, namely, horospheres and equidistants. Below, we state a version of Riemannian Blaschke's \emph{horoball} rolling theorem. We have to restrict to simply connected negatively curved manifolds in order to define horoballs and horospheres in the general Riemannian setting. 

Recall that a complete simply-connected Riemannian manifold of non-positive sectional curvature is called a \textit{Hadamard manifold}. Let $M$ be a Hadamard manifold, and let $\gamma \colon [0, +\infty) \to M$ be a unit-speed geodesic ray emanating from $p = \gamma(0)$. (Recall that a unit speed geodesic $\gamma \colon [0, +\infty) \to M$ is called a \textit{geodesics ray} if $\di(\gamma(s), \gamma(t)) = |t-s|$ for all $t, s \geqslant$ $\gamma$ and $\lim_{t\to \infty} \gamma(t)$ does not exist, where $\di$ is the distance in $M$.) Then the equivalence class of unit-speed geodesic rays $\gamma_1$ such that $\di(\gamma(t), \gamma_1(t))$ is uniformly bounded in $t$ defines an \textit{ideal boundary point} $\gamma(\infty)$ of $M$. Such rays $\gamma_1$ are called \textit{asymptotic to $\gamma$}.

For a ray $\gamma$ in $M$, we define the \textit{Busemann function} $b_\gamma$ \cite[Section 7.3.2]{Pet} as
$$
b_\gamma(q) := \lim \limits_{t \to +\infty} (t - \di(q, \gamma(t))).
$$

In a Hadamard manifold $M$, a ray $\gamma$ emanating from $p = \gamma(0)$ defines a set
$$
H(\gamma) := \{q \in M \colon b_\gamma(q) \geqslant 0\},
$$
which is called a \textit{horoball} passing through $p$ in the direction of $\gamma$. The boundary of $H(\gamma)$ is called a \textit{horosphere}. It contains the point $p$.

Finally, let $M$ be a Hadamard manifold, $\gamma$ be a ray in $M$, and $A \subset M$ be a set. We define an \textit{infinite cone over $A$ with the vertex at $\gamma(\infty)$}, denoted as $\mathcal C(\gamma(\infty), A)$, to be the set of all geodesic rays emanating from a point in $A$ and asymptotic to $\gamma$.

\begin{MainTheorem}[Riemannian Blaschke's horoball rolling theorem]
\label{Thm:convexforhoroballs}
If $M$ be a $C^2$-smooth Hadamard manifold of dimension $m \ge 2$ with $0 > \sec(M) \ge - \lambda^2$, $\lambda > 0$, then every $\lambda$-convex domain $D_\Sigma \subset M$ with $C^3$-smooth boundary can roll freely in a horoball in $M$. 

More precisely, for every point $s \in \Sigma$ if $\gamma_s \colon [0, \infty) \to M$, $\gamma(0)=s$ is the unique geodesic ray starting at $s$ with $\dot \gamma(0)$ being an inward pointing normal, then 
\begin{equation}
\label{Eq:horoballincl}
D_\Sigma \subseteq H(\gamma_s),
\end{equation}
where $H(\gamma_s) \subset M$ is the horoball passing through $s$ in the direction of $\gamma_s$.
\end{MainTheorem}

\begin{remark}
Observe that in the statement of Theorem~\ref{Thm:convexforhoroballs}, $D_\Sigma$ may be non-compact. At the same time, in Theorem~\ref{Thm:convexforballs}, the domain $D_\Sigma$ is necessarily compact.
\end{remark}

As was mentioned in the introduction, the proofs of Theorems~\ref{Thm:convexforballs} and \ref{Thm:convexforhoroballs} are based on two comparison results, the Radial Angle Comparison theorem from \cite{BDr13}, and Toponogov's triangle comparison theorem. In section~\ref{Sec:RAC}, we provide a self-contained treatment of the former comparison theorem, thus filling some of the gaps in the existing literature. We pay a particular attention to the equality case of the Radial Angle Comparison Theorem, which is a novel part. We then combine both comparison theorems to prove Theorems~\ref{Thm:convexforballs} and \ref{Thm:convexforhoroballs} in Sections~\ref{Sec:ProofThmA} are \ref{Sec:ProofThmB} respectively. Finally, in Section \ref{Sec:OpenQuestions} we state some open problems related to Blaschke's rolling theorems in Riemannian manifolds.

%%%%%%%%%%%%%%%%%%%%%%%%%%%%%%%%%%%%%%%%%%%%%%%%%%%%%%%%
%%%%%%%%%%%%%%%%%%%%%%%%%%%%%%%%%%%%%%%%%%%%%%%%%%%%%%%%

\section{The radial angle comparison theorem for $\lambda$-convex domains}
\label{Sec:RAC}

%%%%%%%%%%%%%%%%%%%%%%%%%%%%%%%%%%%%%%%%%%%%%%%%%%%%%%%%
%%%%%%%%%%%%%%%%%%%%%%%%%%%%%%%%%%%%%%%%%%%%%%%%%%%%%%%%

In this section, we state and prove the Radial Angle Comparison theorem for $\lambda$-convex domains in Riemannian manifolds with $\sec \ge c$. We will focus on the case when $\lambda$ satisfies the sphere constraints \eqref{Eq:Sphere}. Some parts of our exposition in this section partly follow~\cite{BDr13} and \cite{DrPhD}.

Let $D_\Sigma \subset M$ be a $\lambda$-convex domain. Assume that $D_\Sigma$ is a convex set, i.e., any two points in $D_\Sigma$ can be joint by a geodesic segment within $D_\Sigma$. For example, this is true when $M$ is $C^2$-smooth and complete and $\Sigma$ is $C^3$-smooth \cite{Convexity}, which we will assume from now on. 

Fix a point $p \in D_\Sigma$, called \emph{the origin}. Define $\Sigma_p = \Sigma \cap \mathring{B}(p, \inj_p(M))$ and $D_{\Sigma_p} = \text{cl}\left(D_\Sigma \cap \mathring{B}(p, \inj_p(M))\right)$ to be the part of $D_\Sigma$ that lies within the open ball of the injectivity radius of $M$ at $p$ (here $\text{cl}(\cdot)$ is the closure of a set). 

For the origin $p$, we can define the \emph{radial angle function} $\phi \colon \Sigma_p \to [0, \pi)$ as follows. Pick a point $q \in \Sigma_p$. There exists a unique distance minimizing geodesic $\gamma_{pq} \colon [0,1] \to M$ such that $\gamma_{pq}(0) = p$, $\gamma_{pq}(1) = q$. Then, by definition, $\phi(q)$ is the smallest angle between $\dot\gamma_{pq}(1)$ and $-\nu(p)$:
$$
\phi(q) := \angle\left(\dot \gamma_{pq}(1), -\nu(q)\right),
$$
where $\nu(q)$ is the inward pointing normal to $\Sigma_p$ at $q$. (The radial angle function depends on the choice of the origin $p$. We will suppress this dependence where it does not cause a confusion.)

Denote by $\dist \colon M \times M \to \mathbb R_{+}$ the distance function on the manifold $M$. We will also use the notation $|pq| := \dist(p,q)$. For the fixed origin point $p \in \Sigma$, define the restricted distance function $\di_p (\cdot) := \dist(p, \cdot)$ with respect to $p$. The restricted distance function $\di_p$ is $C^2$-smooth in some small neighborhood $\mathcal U_p$ of $D_{\Sigma_p}$ punctured at $p$. 

We can define a smooth restriction $\di_{p, \Sigma} \colon \Sigma_p \to \mathbb R_{+}$ by putting 
$$
\di_{p, \Sigma}(q):=\di_p(q) = \dist(p,q) \text{ for any } q \in \Sigma_p.
$$

Since $\di_p$ is smooth in $\mathcal U_p$, we can define a gradient vector field $\grad_M \di_p$, which is well-defined smooth section of $T\mathcal U_p$. Integral trajectories of this vector field are radial geodesics emanating from $p$. Likewise, we can define a gradient vector field $\grad_\Sigma \di_{p, \Sigma}$, which is a smooth section of $T\Sigma_p \subset T\mathcal U_p$.

It is straightforward to see that for each $q \in \Sigma_p$ the both gradients and the radial angle are related by
\begin{equation}
\label{orthogonal}
\grad_M \di_p = \grad_\Sigma \di_{p, \Sigma} - \cos \phi \cdot \nu \quad \text{in} \quad T_qM.
\end{equation}
This decomposition means that the gradient of the restricted distance function is just the orthogonal projection of the gradient of the globally defined distance function onto $T_q \Sigma$ viewed as a subspace of $T_q M$. 

For any $t>0$, denote by $\mathcal S(t)$ the geodesic sphere of radius $t$ centered at $p$. Note that $\grad_M \di_p$ is a unit normal vector field to $\mathcal S(t) \cap \mathcal U_p$. Since $\left|\grad_M \di_p\right| = 1$, from~(\ref{orthogonal}) we get
$$
\left|\grad_\Sigma \di_{p, \Sigma}\right| = \left|\sin \phi \right|.
$$ 
Therefore, $\grad_\Sigma \di_{p, \Sigma}$ vanishes exactly at those points $q$ on $\Sigma_p$ where $\grad_M \di_p (q)$ and $-\nu(p)$ coincide. Equivalently, $\grad_\Sigma \di_{p, \Sigma}$ vanishes at those $q \in \Sigma_p$ where $\Sigma_p$ touches $\mathcal S(|pq|)$. 

Suppose $\phi(q) \neq 0$. Then in a neighborhood of $q$ in $\Sigma_p$ we have a well-defined unit vector field
$$
X := \frac{\grad_\Sigma \di_{p,\Sigma}}{\left|\grad_\Sigma \di_{p,\Sigma}\right|}.
$$
Let $\gamma \colon (-\epsilon, \epsilon) \to \Sigma$ be the unique integral trajectory of $X$ passing through $q$ (see Figure~\ref{Fig:Vectors}):
$$
\gamma(0) = q, \quad \dot\gamma = X.
$$
The curve $\gamma$ can be uniquely extended in both direction until it either hits the possibly non-empty boundary of $\Sigma_p$, or it hits the point where $\phi$ vanishes. We will call such a curve the \emph{maximal integral trajectory}.

\begin{figure}[h]
\definecolor{uququq}{rgb}{0.25,0.25,0.25}
\begin{tikzpicture}[line cap=round,line join=round,>=triangle 45,x=1.0cm,y=1.0cm,scale=1.9]
\clip(-2.3,-2.1) rectangle (1.9,2.2);
\fill[line width=0.4pt,dotted,fill=black,fill opacity=0.05] (-0.06,2.04) -- (1.6,1.78) -- (1.33,0.02) -- (-0.33,0.28) -- cycle;
\draw [shift={(0.64,1.03)},line width=0.4pt,fill=black,fill opacity=0.3] (0,0) -- (27.92:0.12) arc (27.92:81.15:0.12) -- cycle;
\draw [rotate around={38.35:(-0.32,-0.23)},line width=1.2pt] (-0.32,-0.23) ellipse (1.98cm and 1.45cm);
\draw [shift={(-0.43,0.22)},line width=0.4pt,dash pattern=on 1pt off 1pt]  plot[domain=0.23:1.82,variable=\t]({0.95*1.65*cos(\t r)+-0.31*0.69*sin(\t r)},{0.31*1.65*cos(\t r)+0.95*0.69*sin(\t r)});
\draw [dash pattern=on 1pt off 1pt] (-0.64,0.36)-- (0.64,1.03) node[midway, below, xshift=8pt, yshift=-2pt] {$t\!=\!\di_p(q)$};
\draw [->] (0.64,1.03) -- (0.75,1.74) node[left] {$-\nu$};
\draw [->] (0.64,1.03) -- (1.27,1.37) node[above, yshift=-2pt] {$\grad_M \di_p$};
\draw [->] (0.64,1.03) -- (1.34,0.92) node[above right, xshift=-2pt, yshift=-2pt] {$X$};
\draw [->] (0.64,1.03) -- (0.97,0.4) node[right] {$Y$};
\draw [line width=0.4pt,dotted] (-0.06,2.04)-- (1.6,1.78);
\draw [line width=0.4pt,dotted] (1.6,1.78)-- (1.33,0.02);
\draw [line width=0.4pt,dotted] (1.33,0.02)-- (-0.33,0.28);
\draw [line width=0.4pt,dotted] (-0.33,0.28)-- (-0.06,2.04);
\draw [line width=0.4pt,fill=black,fill opacity=0.02] (-0.64,0.36) circle (1.44cm);
\draw (0.8,-1.4) node[] {$\Sigma$};
\draw (0.1,-1.1) node[] {$\mathcal S(t)$};
\fill [color=black] (-0.64,0.36) circle (1.0pt) node[below] {$p$};
\draw[color=black] (-0.3,1.06) node {$\gamma$};
\fill [color=black] (0.64,1.03) circle (1.0pt) node[below, yshift=-2pt, xshift=-2pt] {$q$};
\draw[color=black] (0.78,1.21) node {$\phi$};
\end{tikzpicture}
\caption{The radial angle and the associated vector fields.}
\label{Fig:Vectors}
\end{figure}
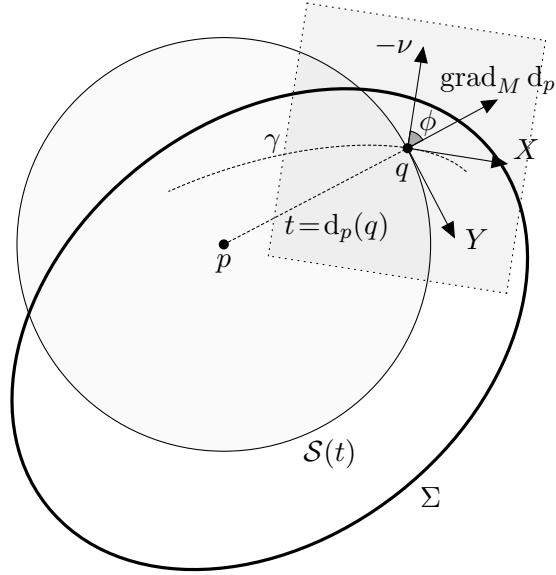

Denote by $s$ the parameter on the maximal integral trajectory $\gamma$ for $X$. Since $X$ is unit, $s$ is an arc-length parameter. We want to consider re-parametrization of $\gamma$, namely, suppose
\begin{equation}
\label{eqtpar} %t parameter
t(s) := \di_{p, \Sigma} (\gamma(s)).
\end{equation}
Since $\grad_\Sigma \di_{p, \Sigma}$ doesn't vanish along $\gamma$, we can change our parametrization from $s$ to $t$ on the whole maximal integral trajectory.

Introduce an auxiliary unit vector field $Y$ such that 
$$
\gamma(t) \mapsto Y(t) \in \spa\{\nu,\grad_M \di_p, X\}_{\gamma(t)} \subset T_{\gamma(t)}M,
$$
$Y(\gamma(t))$ is perpendicular to $\grad_M \di_p (\gamma(t))$ and making the angle $\phi(\gamma(t))$ with $X(\gamma(t))$ (see Figure~\ref{Fig:Vectors}). Therefore, $Y$ is a well-defined along $\gamma$ vector field that is tangent to $\mathcal S(t)$. 

The following formula, due to Borisenko, is a multidimensional generalization of Liouville's type formula~\cite[Ch. 4-4, Prop. 4]{doC76}, and, to best of our knowledge, was first obtained in~\cite{BGR01} (see also~\cite[Lem.~2.2]{AB98}, where the authors study a similar construction).

\begin{proposition}[\cite{BGR01, Bor02}]
\label{borlem}
In the notations above, if $\gamma(t) \subset \Sigma_p$ is an integral trajectory of the vector field $X$ parameterized by the distance parameter $t$ from $p$, then for any $t$ at the point $q = \gamma(t)$ we have the identity
\begin{equation}
\label{borform}
k^{\nu} \left(q, X\right) =  \mu^{-\grad_M \di_p}\left(q,Y\right) \cdot \cos \phi  + \frac{d}{dt} \cos \phi,
\end{equation}
where $\mu^{-\grad_M \di_{p}}(q,Y)$ is the normal curvature of the sphere $\mathcal S(t)$ at the point $q$ in the direction $Y(q)$ with respect to the unit normal $-\grad_M \di_p (q)$. 
\end{proposition}

\begin{proof}
The proof of this fact that was given in \cite{BGR01} and \cite{Bor02} contains the same computational inaccuracy. Because of that, we provide a complete proof here.

 Suppose $\nabla$ is a Riemannian connection on $M$. We can smoothly extend both unit vector fields $X$ and $Y$ in some neighborhood of $q$ in $M$; we will keep the same notation. The result does not depend on the extension. For brevity, denote $Z := \grad_M \di_p$.

From the Gauss decomposition it follows that 
\begin{equation}
\label{preqnormc}
k^\nu (q, X) = \left<\nabla_X X, \nu\right>, \quad \mu^{-Z} (q, Y) = -\left<\nabla_Y Y, Z\right>
\end{equation}
for unit $X$ and $Y$.
  
We proceeds with a straightforward calculation. By the construction of vector fields we have the following vector decompositions:

\begin{equation}
\label{riemdec}
\begin{aligned}
X &=  \sin \phi \cdot Z + \cos \phi \cdot Y \\
\nu &= - \cos \phi \cdot Z + \sin \phi \cdot Y .
\end{aligned}
\end{equation}

Using~(\ref{preqnormc}) and~(\ref{riemdec}) we compute:
\begin{equation*}
\begin{aligned}
k^{\nu} \left(q, X\right) &= \left<\nabla_X X, \nu\right> \\
&= \left(\cos \phi  \left<Z, \nu \right> - \sin \phi  \left<Y,\nu\right>\right)  X(\phi)  + \sin \phi \left<\nabla_X Z, \nu\right> + \cos \phi \left<\nabla_X Y, \nu\right> \\
&= -X(\phi) +  \cos^2 \phi \left<\nabla_Y Y, \nu\right> + \cos \phi \sin \phi \left( \left<\nabla_Y Z, \nu \right> + \left<\nabla_Z Y, \nu\right>\right)  
\end{aligned}
\end{equation*}
(here we used the fact that the integral curves for $Z$ are geodesics, and thus $\nabla_Z Z = 0$).

Recall that $Z$ and $Y$ are unit vector fields, and thus $\left<\nabla_W Y, Y\right>=\left<\nabla_W Z, Z\right>=0$ for any vector $W \in T_q \Sigma$. Substituting $\nu$ from~(\ref{riemdec}) we obtain:
\begin{equation*}
\begin{aligned}
k^{\nu} \left(q, X\right) &= - X(\phi) - \cos^3 \phi \left<\nabla_Y Y, Z\right> \\&+ \cos \phi \sin \phi \left( \sin \phi \left<\nabla_Y Z, Y \right>  - \cos \phi \left<\nabla_Z Y, Z\right>\right).
\end{aligned}
\end{equation*}

We can arrange at extension in a neighborhood of $q$ in $M$ of the initial vector field $Y$ in such a way that $\left<Y, Z\right>=0$ (again, we keep the same notation for the extension). Then $\left<\nabla_Z Y, Z\right> = \left<Y, \nabla_Z Z\right> = 0$ and 
$\left<\nabla_Y Z, Y \right> = -\left<Z, \nabla_Y Y\right>$. Altogether we obtain:
\begin{equation*}
k^{\nu} \left(q, X\right) = - X(\phi) - \cos \phi \left<\nabla_Y Y, Z\right> = X(\phi) + \mu^{-Z}(q, Y) \cdot \cos \phi.
\end{equation*}

Finally, observe that $$X(\phi) = X(t) \cdot \frac{d\phi}{dt} = \sin \phi \cdot \frac{d\phi}{ dt} = - \frac{d}{dt} \cos \phi$$ because of ~(\ref{eqtpar}) and the first decomposition in~(\ref{riemdec}). Proposition~\ref{borlem} is proved.
\end{proof}

Proposition~\ref{borlem} allows us to establish the following comparison-type result that we call the \emph{Radial Angle Comparison Theorem} (compare Figure~\ref{Fig:RAC}).

\begin{theorem}[Radial Angle Comparison Theorem for $\lambda$-convex domains]
\label{RACThm}
Let $M$ be a $C^2$-smooth complete Riemannian manifold of dimension $m \ge 2$ and with 
$$
\sec(K) \geqslant c.
$$
Suppose that $\lambda \in \mathbb R$ satisfies the sphere constraints \eqref{Eq:Sphere}, and assume that $\inj(M) > 2R_\lambda$ for the radius $R_\lambda$ corresponding to $\lambda$. Let $D_\Sigma \subset M$ be a $\lambda$-convex domain with $C^3$-smooth boundary and $D_{\mathcal S_\lambda} \subset M(c)$ be a ball of radius $R_\lambda$, with $\mathcal S_\lambda := \partial D_{\mathcal S_\lambda}$. Choose a point $p \in D_\Sigma \sm \Sigma$ at distance $\dist(p, \Sigma) = |ps| < \inj_p(M)$ for some $s \in \Sigma$, and let $\phi \colon \Sigma_p \to [0, \pi)$ be the radial angle function defined on $\Sigma_p$ with respect to $p$. Furthermore, let $p_\lambda \in D_{\mathcal S_\lambda} \setminus \mathcal S_\lambda$ be a point such that 
$$
\dist(p, \Sigma) = \dist(p_\lambda, \mathcal S_\lambda),
$$
and let $\phi_\lambda \colon \mathcal S_\lambda \to [0, \pi)$ be the radial angle function with respect to $p_\lambda$.

Then for all pairs of points $q \in \Sigma_p$, $q_\lambda \in \mathcal S_\lambda$ with 
$$
|pq| = |p_\lambda q_\lambda|,
$$
for the radial angles satisfy
\begin{equation}
\label{actheq1}
\phi(q) \leqslant \phi_\lambda(q_\lambda).
\end{equation}

Moreover, if we have equality in~(\ref{actheq1}), then the maximal integral trajectory $\gamma \colon [0,1) \to \Sigma_p$ of the distance function $\di_{p, \Sigma} \colon \Sigma \to (0, +\infty)$ with $q = \gamma(0)$ terminates at some $s_1 = \lim_{r \to 1-0} \gamma(r)$ with $|ps_1| = |ps|$, and the cone $\mathcal C(p, \gamma[0,1])$ over $\gamma[0,1] = \gamma[0,1) \cup s_1$ with the vertex at $p$ is a totally geodesic 2-dimensional submanifold in $M$ isometric to a curvilinear triangle in a disk of radius $R_\lambda$ in $M^2(c)$ bounded by two geodesics and an arc of the boundary of the disk. 

\end{theorem}

\begin{figure}[h]
\begin{tikzpicture}[line cap=round,line join=round,>=triangle 45,x=1.0cm,y=1.0cm, scale = 1]
\draw (-2.78,2.6) node[anchor=north west] {$M$};
\draw (4.0,2.6) node[anchor=north west] {$M(c)$};
\clip(-3.06,-3.03) rectangle (11.46,2.14);
\draw [shift={(1.38,0.74)},fill=black,fill opacity=0.1] (0,0) -- (22.34:1.29) arc (22.34:48.04:1.29) -- cycle;
\draw (2.5,1.5) node[xshift=14pt, yshift=6pt] {$\phi(q)$};
\draw [shift={(8.46,0.03)},fill=black,fill opacity=0.1] (0,0) -- (17.69:1.29) arc (17.69:53.16:1.29) -- cycle;
\draw (10, 0.5) node[above right, xshift=-14pt, yshift=6pt] {$\phi_\lambda(q_\lambda)$};
\draw [rotate around={38.35:(-0.32,-0.23)},line width=1.2pt] (-0.32,-0.23) ellipse (1.98cm and 1.45cm);
\draw [line width=1.2pt] (6.48,-0.6) circle (2.07cm);
\draw [->] (1.38,0.74) -- (2.94,1.38) node[below, yshift=-2pt] {$-\nu$};
\draw [->] (1.38,0.74) -- (2.51,1.99);
\draw [->] (8.46,0.03) -- (9.47,1.38);
\draw [->] (8.46,0.03) -- (10.07,0.55) node[below, yshift=-2pt] {$-\nu_\lambda$};
\draw [dash pattern=on 1pt off 1pt] (7.14,-1.73)-- (7.54,-2.38);
\draw [dash pattern=on 1pt off 1pt] (7.14,-1.73)-- (8.46,0.03) node[below right] {$q_\lambda$};
\draw [dash pattern=on 1pt off 1pt] (-0.09,-0.9)-- (0.31,-1.55);
\draw [dash pattern=on 1pt off 1pt] (-0.09,-0.9)-- (1.38,0.74) node[below right, yshift=1pt] {$q$};
\draw (0.2,-1.6) node[anchor=north west] {$s$};
\draw (4.59,-0.3) node[anchor=north west] {$\mathcal S_\lambda$};
\draw (-1.91,0.05) node[anchor=north west] {$\Sigma$};
\fill [color=black] (-0.09,-0.9) circle (1.0pt) node[left] {$p$};
\fill [color=black] (0.31,-1.55) circle (1.0pt);
\fill [color=black] (1.38,0.74) circle (1.0pt);
\fill [color=black] (7.54,-2.38) circle (1.0pt);
\fill [color=black] (7.14,-1.73) circle (1.0pt) node[left] {$p_\lambda$};
\draw[color=black] (7.8,-2.6) node {$s_\lambda$};
\fill [color=black] (8.46,0.03) circle (1.0pt);
\end{tikzpicture}
\caption{The Radial Angle Comparison Theorem.}
\label{Fig:RAC}
\end{figure}
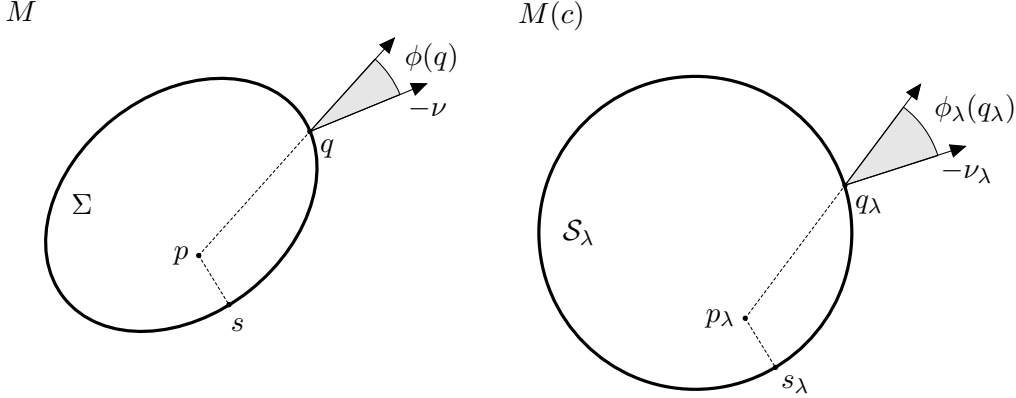

\begin{proof}
Let $s_\lambda \in \mathcal S_\lambda$ be a point such that $|p_\lambda s_\lambda| = \dist(p_\lambda, \mathcal S_\lambda) = |ps| = \dist(p,\Sigma) =: d$. By construction, $\phi(s) = \phi_\lambda (s_\lambda) = 0$.

We want to apply Proposition~\ref{borlem}. Note that the gradient $\grad_{M(c)} \di_{p_\lambda}$ vanishes on $\mathcal S_\lambda$ only at $s_\lambda$ and at the point diametrically opposite to $s_\lambda$, that is at the distance $2 R_\lambda - d$ from $p_\lambda$.

Assume first that $|pq| \neq 2 R_\lambda - d$. Then $\grad_{M(c)} \di_{p_\lambda}(q_\lambda) \neq 0$. And thus if $\grad_{M} \di_{p}(q) = 0$, then in~(\ref{actheq1}) we have a strict inequality of the form $0 < \phi(q_\lambda)\neq 0$, which is true. Therefore, we can assume $\grad_{M} \di_{p}(q) \neq 0$. We are in position of applying Proposition~\ref{borlem}.

Denote $\gamma(t)$ and $\gamma_\lambda(t)$ the maximal integral trajectories of the gradient vector fields given by Proposition~\ref{borlem} passing though $q$ and $q_\lambda$ on $\Sigma$ and $\mathcal S_\lambda$, respectively. Note that the limit points of $\gamma_\lambda$ are precisely $s_\lambda$ and its diametrically opposite point. 

By Proposition~\ref{borlem}, along $\gamma$ and $\gamma_\lambda$ we have
\begin{equation}
\label{normcF}
k^{\nu}(t) = \cos \phi (t) \cdot \mu^{\grad_M \di_p}(t) + \frac{d}{dt} \cos \phi (t),
\end{equation}
\begin{equation}
\label{normcF1}
\lambda = \cos \phi_\lambda (t) \cdot \mu^{\grad_{M(c)} \di_{p_\lambda}} (t) + \frac{d}{dt}\cos \phi_\lambda (t),
\end{equation}
where, for brevity, (in the notations from Proposition~\ref{borlem}) $k^{\nu}(t) = k^{\nu}(\gamma(t), X)$, $\mu^{\grad_M \di_p}(t) = \mu^{\grad_M \di_p}(\gamma(t), Y)$, and $\mu^{\grad_{M(c)} \di_{p_\lambda}} (t)$ is a normal curvature of a sphere of radius $t$ in the model space $M(c)$. 

Assumptions of the theorem allows us to use the comparison theorem for normal curvatures geodesic spheres (see \cite[Chapter IV, Lemma~2.9]{Sa92}). In our notations it yields
\begin{equation}
\label{normcsphcomp}
\ct_c (t) = \mu^{\grad_{M(c)} \di_{p_\lambda}} (t) \leqslant \mu^{\grad_M \di_p}(t).
\end{equation}

Let us subtract from~(\ref{normcF}) equality~(\ref{normcF1}); then using~(\ref{normcsphcomp}) and the $\lambda$-convexity condition $k^{\nu}(t) \geqslant \lambda$, we obtain
\begin{equation}
\label{mainineq2}
0  \leqslant k^{\nu}(t) - \lambda  \leqslant \left(\cos \phi(t) - \cos \phi_\lambda(t) \right) \cdot \ct_c (t)  + \frac{d}{dt} \left(\cos \phi(t) - \cos \phi_\lambda(t) \right).
\end{equation}

If we set $f(t) := \cos \phi (t) - \cos \phi_\lambda(t)$, then it follows from~(\ref{mainineq2}) that this function satisfies the following differential inequality 
\begin{equation}
\label{difineq1}
 f(t) \ct_c(t)  + \frac{d}{dt} f(t) \geqslant 0.
\end{equation}
Obviously, inequality~(\ref{difineq1}) is equivalent to
\begin{equation*}
\label{difineq2}
\frac{d}{dt}\left( f(t)\sn_c(t)\right) \geqslant 0.
\end{equation*}

Therefore, the function $f\cdot\sn_c$ is monotonically increasing. Going along $\gamma(t)$ following decreasing $t$ we will eventually arrive at a limit point, say $s^* \in \Sigma$, corresponding to some distance $t^*$ from $p$. Note that $t^* \geqslant d$ since $d$ is a strong minimum of the distance function $\di_p(\cdot)$. Clearly, $\phi(s^*) = 0$ by definition of $s^*$. Moreover, $\gamma_\lambda(t^*)$ is either equal to $s_\lambda$ with $t^*=d$, or $\gamma_\lambda(t^*)$ is an interior point of the corresponding maximal trajectory. In the first case, $f(t^*) = 0$, in the second $f(t^*) > 0$. 

In either case, $f(t^*) \geqslant 0$, and thus monotonicity of $f \cdot \sn_c$ yields~(\ref{actheq1}) at $q$ and $q_\lambda$, as claimed.

To finish the proof of inequality (\ref{actheq1}) it is left to consider the case when $|pq| = 2R_\lambda - d$. In such case $\grad_{M(c)} \di_{p_\lambda} (q_\lambda) = 0$. If $\grad_M \di_p(q) = 0$ we are done. If $\grad_M \di_p(q) \neq 0$, then there exist a maximal trajectory $\gamma(t)$ through $p$. Consider a point $q^\epsilon := \gamma(2 R_\lambda - d - \epsilon)$ for some small $\epsilon$. Then we are under conditions of the proven case, and thus obtain $\varphi(q^\epsilon) \leqslant \varphi(q^\epsilon_\lambda)$. Sending $\epsilon \to 0$ we recover $\varphi(q) \leqslant \varphi(q_\lambda) (= 0)$ by continuity, which finishes the proof of (\ref{actheq1}). 

Let us treat the equality case. Assume that $\phi(q) = \phi_\lambda(q_\lambda)$. Then since $f(t) \cdot \sn_c(t) = (\cos \phi(t) - \cos \phi_\lambda(t)) \sn_c(t)$ with $\sn_c(t) \not\equiv 0$ is a monotonically increasing function, and moreover $f(t^*) = f(|pq|) = 0$, it follows that $f(t) \equiv 0$ along the trajectory $\gamma$ between $q$ and $s^*$. But, by construction, $\phi(t^*) = 0$; hence $\phi_\lambda(t^*) = 0$, which is only possible at $s_\lambda$. Therefore $t^* = d$. Putting $s_1:=s^*$, we obtain a desired point. Let us show that the cone $\mathcal C(p, \gamma[d, |pq|])$ is isometric to a sector of a disk with radius $R_\lambda$ in $M^2(c)$.

Indeed, since $f(t) \equiv 0$ for all $t \in \left[d, |pq|\right]$, we must have equality in~(\ref{mainineq2}) for all such $t$. Taking into account~(\ref{normcsphcomp}), this is only possible when
\begin{equation}
\label{Eq:equalitycase1}
\mu^{\grad_M \di_p}(t) = \ct_c(t)
\end{equation}
and 
\begin{equation}
\label{Eq:equalitycase2}
k^\nu(t) \equiv \lambda \quad \text{for all } t \in \left[d, |pq|\right].
\end{equation}

For $t \in \left[d, |pq|\right]$, let $\delta_t \colon s \mapsto \delta_t(s)$, $s \in [0, t]$,  be the unique geodesic segment connecting $p = \delta_t(0)$ and $\gamma(t) = \delta_t(t)$ (we assume the arc length parametrization of $\delta_t$). Recall (in the notation of Proposition~\ref{borlem}) that $Y(t)$ was a unit vector field defined along $\gamma(t)$ by lying in $\spn\{\gamma'(t), \grad_M \di_p(\gamma(t))\}$ and being perpendicular to $\grad_M \di_p (\gamma(t))$. Consider the unique Jacobi field $J(s) \in T_{\delta_t(s)} M$ along $\delta_t$ such that $J(0) = 0$ and $J(t) = Y(t)$.

By the equality case in the Rauch comparison theorem (see~\cite[Chapter IV, Theorem 2.3]{Sa92}), for which the normal curvature comparison~(\ref{normcsphcomp}) is a corollary (see \cite[Chapter IV, Lemma 2.9 and Theorem 2.7]{Sa92}, we obtain that the sectional curvature $K_\sigma(x)$ of $M$ satisfies
\begin{equation}
\label{Eq:constcurv}
K_{\sigma(s)}(\delta(s)) \equiv c,
\end{equation}
where 
$$
\sigma(s) := \spn\left\{d\delta_t/ds |_s, J(s)\right\} \subset T_{\delta_t(s)}M
$$
is a 2-dimensional plane spanned by $d\delta_t / ds |_s$ and $J(s)$.

Put $\tilde \gamma(t) := \exp^{-1}_p \gamma(t)$, and define $\delta(t,s):=\exp_p (s \tilde \gamma(t))$, $(t, s) \in [d, |pq|] \times [0,t]$. Then $\delta$ defines a smooth 2-dimensional submanifold in $M$, which we called a cone $\mathcal C(p, \gamma)$, such that for a given $t$ the curve $\delta(t, s) = \delta_t(s)$ is a geodesic. The vector field $I(s):=\partial \delta/ \partial t (t,s)$ is a Jacobi field along $\delta_t$. Moreover, $I(0)=0$ and $I(t)$ is proportional to $X(t)$, where $X(t)$ is a unit vector field tangent to $\gamma(t)$. By construction, $X(t) \in \spn\{Y(t), \grad_M \di_p(\gamma(t))\}$. Therefore, $I(t) \in \spn \{J(t), \grad_M \di_p(\gamma(t))\} = \spn\{J(t), d\delta_t/ds|_{s=t}\}$. If $E_1(s) = d \delta/ds (s)$, $E_2(s), \ldots, E_m(s)$ is the orthonormal frame spanning $T_{\delta_t(s)} M$ and parallel along $\delta_t(s)$ such that $\spn\{E_1(t), E_2(t)\} = \sigma(t)$, then it follows that $J(s) \in \spn\{E_1(s), E_2(s)\}$ and $I(s) \in \spn\{E_1(s), E_2(s)\}$ for all $s \in [0,t]$. This implies that $\sigma(s) = \spn\{d \sigma_t /ds |_s, I(s)\} = \{\partial \delta /\partial s, \partial \delta / \partial t\}|_{(t,s)}$. Using (\ref{Eq:constcurv}) we obtain that $\mathcal C(p, \gamma)$ is a 2-dimensional submanifold of constant curvature $c$ bounded by two geodesics and a curve of constant geodesic curvature $\lambda$. The last assertion follows from~(\ref{Eq:equalitycase2}). The theorem is proven.  
\end{proof}

%%%%%%%%%%%%%%%%%%%%%%%%%%%%%%%%%%%%%%%%%%%%%%%%%%%%%%%%
%%%%%%%%%%%%%%%%%%%%%%%%%%%%%%%%%%%%%%%%%%%%%%%%%%%%%%%%
\section{Proof of Theorem~\ref{Thm:convexforballs}}
\label{Sec:ProofThmA}
%%%%%%%%%%%%%%%%%%%%%%%%%%%%%%%%%%%%%%%%%%%%%%%%%%%%%%%%
%%%%%%%%%%%%%%%%%%%%%%%%%%%%%%%%%%%%%%%%%%%%%%%%%%%%%%%%

Consider an arbitrary point $s \in \Sigma$, and let $\bar s$ and $B(\bar s, R_\lambda)$ be as in the statement of the theorem. Note that the point $\bar s$ with $|s \bar s| = R_\lambda$ is well defined since $R_\lambda < \inj (M)$. For brevity, put $S:=\partial B(\bar s, R_\lambda)$.

\medskip
\noindent{\textbf{I. Rolling:}} In order to prove~(\ref{Eq:ballincl}) it is enough to show
\begin{equation}
\label{racbltheq-1}
|\bar s q| \leqslant R_\lambda \quad \text{ for all } \quad q \in \Sigma \sm \{s\}. 
\end{equation}

For some small $\epsilon < R_\lambda$, pick a point $p \in D_\Sigma$ such that $\epsilon = \dist(p,\Sigma) = |pq|$. Such a point always exists. 

We want to apply the Radial Angle Comparison Theorem (Theorem~\ref{RACThm}). For that, let $\mathcal S_\lambda$ be a sphere of curvature $\lambda$ in the model space $M(c)$ centered at $\bar s_\lambda$, $D_{\mathcal S_\lambda}$ be the ball bounded by $\mathcal S_\lambda$, and consider a point $p_\lambda \in D_{\mathcal S_\lambda} \sm \mathcal S_\lambda$ such that $\dist (p_\lambda, \mathcal S_\lambda) =  \epsilon$ (see Figure~\ref{Fig:Comparison}). Then for the origins $p$ and $p_\lambda$ in $D_\Sigma$ and $D_{\mathcal S_\lambda}$ we have well-defined radial angle functions $\phi$ and $\phi_\lambda$, correspondingly.  

\begin{center}
\begin{figure}[h]
\makebox[\textwidth][c]{\includegraphics[width=1.2\textwidth, trim=40 22 30 33, clip]{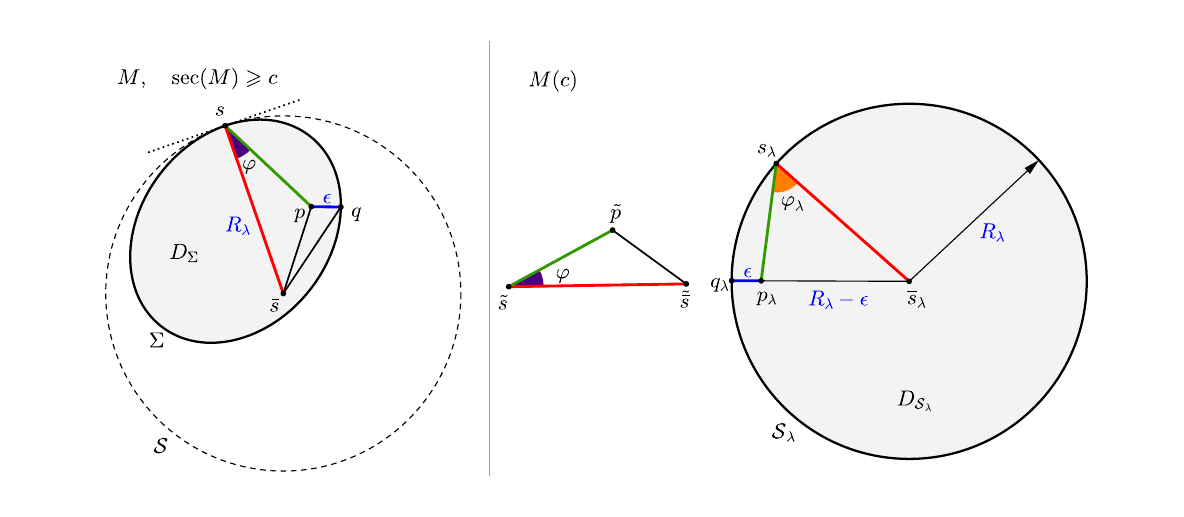}}
\caption{Comparison construction in the proof of Theorem~\ref{Thm:convexforballs}.}
\label{Fig:Comparison}
\end{figure}
\end{center}

There are two possibilities: either $|ps| < 2R_\lambda - \epsilon$, or $|ps| \geqslant 2 R_\lambda - \epsilon$. First assume $|ps| < 2R_\lambda - \epsilon$; then there exists a point $s_\lambda \in \mathcal S_\lambda$ such that 
\begin{equation}
\label{racbltheq0}
|p s| = |p_\lambda s_\lambda|.
\end{equation}

By the Radial Angle Comparison Theorem (Theorem~\ref{RACThm}), 
\begin{equation}
\label{racbltheq1}
\phi (s) \leqslant \phi_\lambda(s_\lambda).
\end{equation}

Denote by $\Delta$ the geodesic triangle $\triangle \bar s s p$ in the Riemannian manifold $M$, and by $\Delta_\lambda$ the geodesic triangle $\triangle \bar s_\lambda s_\lambda p_\lambda$ in $M(c)$. Finally, in the model space $M(c)$ introduce a comparison triangle $\tilde \Delta = \triangle \tilde{\bar s} \tilde s \tilde p$ such that 
$$
\left|\tilde{\bar s} \tilde s\right| = \left|\bar s s\right| \,(=R_\lambda), \quad \left|\tilde s \tilde p\right| = |s p|, \quad \angle \tilde{\bar s} \tilde s \tilde p = \angle \bar s s p \, (=\phi(s)) .	
$$

By Toponogov's triangle comparison theorem (see, e.g., \cite[Chapter IV, Thm.~4.2 (2)]{Sa92}) for the triangles $\Delta$ and $\tilde \Delta$ we have
\begin{equation}
\label{racbltheq2}
|\bar s p| \leqslant |\tilde{\bar s} \tilde p|.
\end{equation}

At the same time, from~(\ref{racbltheq0}), (\ref{racbltheq1}) and the way we have constructed the point $\bar s_\lambda$ it follows that for the elements of the triangles $\tilde \Delta$ and $\Delta_\lambda$ in $M(c)$ the relations
\begin{equation*}
\begin{aligned}
&R_\lambda= |\bar s_\lambda s_\lambda| = \left|\tilde{\bar s} \tilde s\right|, \quad |s_\lambda p_\lambda| = |\tilde s \tilde p|, \\
&\phi_\lambda(s_\lambda) = \angle \bar s_\lambda  s_\lambda p_\lambda \geqslant \angle \tilde{\bar s} \tilde s \tilde p =\phi(s)
\end{aligned}
\end{equation*}
hold. Therefore,
\begin{equation}
\label{Eq:racineq}
|\bar s_\lambda p_\lambda| \geqslant \left|\tilde {\bar s} \tilde p\right|.
\end{equation}
Taking into account inequality~(\ref{racbltheq2}) we arrive to the following key inequality
\begin{equation*}
\label{racbltheq3}
|\bar s p| \leqslant |\bar s_\lambda p_\lambda|.
\end{equation*}

By construction, $|\bar s_\lambda p_\lambda| = R_\lambda - \dist(p_\lambda, \mathcal S_\lambda) = R_\lambda - \dist(p, \Sigma) = R_\lambda - \epsilon$. Thus
\begin{equation}
\label{racbltheq4}
|\bar s p| \leqslant R_\lambda - |pq|.
\end{equation}  
From the geodesic triangle $\triangle \bar s p q$ by the triangle inequality and inequality~(\ref{racbltheq4}), it follows that
\begin{equation}
\label{Eq:triangleineq}
|\bar s q| \leqslant |\bar s p| + |p q| \leqslant R_\lambda - |pq| + |pq| = R_\lambda,
\end{equation}
which proves~(\ref{racbltheq-1}) under the assumption $|ps| < 2 R_\lambda - \epsilon$.

It remains to consider what will happen if $|ps| \geqslant 2R_\lambda - \epsilon$ for any small $\epsilon$. In such situation the radial angle comparison theorem is not directly applicable. In fact, such situation is not possible at all, as we will see momentarily. 

If $|sq| < 2R_\lambda$ for any $q \in \Sigma$, then there exits $\delta_q = \delta(q) > 0$ such that $|sq| = 2 R_\lambda - \delta_q$. Thus if we choose $\epsilon < \delta_q / 2$ and $p$ as above ($|pq| = \epsilon$), then
$$
|sp| \leqslant |sq| + |qp| = 2 R_\lambda - \delta_q + |pq| < 2R_\lambda - \epsilon, 
$$
and we can run the argument above. 

Obviously, for $q$ in the neighborhood of $s$ on $\Sigma$ the inequality $|sq| < 2R_\lambda$ holds. Let $U(s) \subset \Sigma$ be the maximal such neighborhood. Then $|sb| = 2 R_\lambda$ for every $b \in \partial U(s)$, if the boundary is not empty. By the argument above, in particular from~(\ref{racbltheq-1}), $U(s) \subset B(\bar s, R_\lambda)$. The geodesic sphere $\mathcal S$, as a subset of $M$, has diameter $2R_\lambda$, hence $\partial U(s)$ can be only a single point diametrically opposite to $s$. Claim~(\ref{Eq:ballincl}) follows.

\medskip
\noindent{\textbf{II. Equality case:}} Let us now turn to the equality case. Recall that $\Sigma_s = \Sigma \cap B(\overline s, R_\lambda)$. It is clear that $s \in \Sigma_s$; assume that $\Sigma_s \sm \{s\} \neq \emptyset$. Pick a point $q \in \Sigma_s \sm \{s\}$, and assume that $q$ is not diametrically opposite to $s$. Then at $q$ we must have equalities simultaneously in (\ref{racbltheq2}), (\ref{Eq:racineq}) and (\ref{Eq:triangleineq}). Let us analyze each equality condition.

Equality in (\ref{racbltheq2}) implies equality case in Toponogov's comparison theorem (see \cite[Chapter~IV, Remark~4.5]{Sa92}) which tells that $\triangle \bar s s p$ must span a totally geodesic 2-dimensional submanifold isometric to a geodesic triangle in the model space of curvature $c$.

Equality in (\ref{Eq:racineq}) yields equality in the Radial Angle Comparison Theorem (Theorem~\ref{RACThm}). By decreasing $\epsilon$ we may assume that distance function $\di_p$ has a unique minimum at $q$, and therefore the integral trajectory $\gamma$ of $\di_{p, \Sigma}$ passing through $s$ will have $q$ as a limiting point. Applying the equality case in Theorem~\ref{RACThm} to the arc $\gamma_{sq}$ of this trajectory between $s$ and $q$ we obtain that the cone $\mathcal C(p, \gamma_{sq})$ must be a 2-dimensional totally geodesic submanifold of $M$ isometric to a 2-dimensional curvilinear triangle in $M(c)$ bounded by two geodesic segments and an arc of constant geodesic curvature $\lambda$. 

Finally, equality in (\ref{Eq:triangleineq}) means that we must have equality in a triangle inequality, which is only possible if the three points $q$, $p$ and $\bar s$ lie on the same geodesic. 

Combining all the equality conditions above we get the following \textsc{Claim}: if $q \in \Sigma_s \sm \{s\}$, then the cone $\mathcal C(\bar s, \gamma_{sq})$ (which, as a set, is equal to $\Delta \cup \mathcal C(p, \gamma_{sq})$) is a totally geodesic 2-dimensional submanifold of $M$ isometric to a sector in a 2-dimensional disk of radius $R_\lambda$ lying in $M(c)$. 

This \textsc{Claim} yields that $|\bar s x| = R_\lambda$ for every $x \in \gamma_{sq}$. Therefore, $\gamma_{sq} \subset \Sigma_s$ for every $q \in \Sigma_s \sm \{s\}$ , and thus $\Sigma_s$ is connected.

Let $\bar s_\lambda$ be a point in $M(c)$, and fix some linear isometry $I \colon T_{\bar s} M \to T_{\bar s_\lambda} M(c)$ between the tangent spaces of $M$ and $M(c)$ at $\bar s$ and $\bar s_\lambda$. Define a map 
$$
\Phi := \exp_{\bar s_\lambda} \circ I \circ \exp^{-1}_{\bar s} \colon M \to M(c).
$$
Since $R_\lambda < \inj(M)$, the map $\Phi$ sends $B(\bar s, R_\lambda)$ diffeomorphically onto $D_{\mathcal S_\lambda}$. Let $\gamma_{s_\lambda q_\lambda}$ be a unique geodesic segment in $\mathcal S_\lambda$ joining $s_\lambda = \Phi(s)$ and $q_\lambda = \Phi(q)$. Then, by definition of a cone, $\mathcal C(\bar s_\lambda, \gamma_{s_\lambda q_\lambda})$ is a 2-dimensional sector in $D_{\mathcal S_\lambda}$ with the angle less than $\pi$. By the \textsc{Claim} above, $\Phi$ maps $\mathcal C(\bar s, \gamma_{sq})$ isometrically onto $\mathcal C(\bar s_\lambda, \gamma_{s_\lambda q_\lambda})$. Hence, the length of $\gamma_{sq}$ is less than a half of circumference of $\mathcal S_\lambda$.

Let $S \subset T_{\bar s} M$ be an $m$-dimensional sphere with radius $R_\lambda$ centered at the origin. Define $\conv_S(A)$ to be a convex hull in $S$ of a subset $A \subset S$; we regard $T_{\bar s} M$ as $\mathbb R^m$ and $S$ as a sphere with the standard metric in $\mathbb R^m$. In such terms for a given pair $q_1, q_2 \in \Sigma_s$ we showed that 
\begin{equation}
\label{Eq:convcond}
\exp_{\bar s}(x) \in \Sigma_s \text{ for all }x \in \conv_S \left(\{\exp_{\bar s}^{-1}(q_1), \exp_{\bar s}^{-1}(q_2)\}\right).
\end{equation}

Applying (\ref{Eq:convcond}) inductively, we get that $\Phi(\Sigma_s)$, as a subset of $\mathcal S_\lambda$, must be convex. This yields the equality case under conditions of case (ii). Otherwise, if we are under the conditions of case (i) of Theorem~\ref{Thm:convexforballs}, then 
$$
\conv_S(\{\exp_{\bar s}^{-1}(s_1), \ldots, \exp_{\bar s}^{-1}(s_{m+1}) \}) = S,
$$
which immediately implies that $\Sigma_s = \partial B(\bar s, R_\lambda)$, and the conclusion of case (i) follows. The theorem is proven.
\qed

%%%%%%%%%%%%%%%%%%%%%%%%%%%%%%%%%%%%%%%%%%%%%%%%%%%%%
\section{Proof of Theorem~\ref{Thm:convexforhoroballs}}
\label{Sec:ProofThmB}
%%%%%%%%%%%%%%%%%%%%%%%%%%%%%%%%%%%%%%%%%%%%%%%%%%%%%

The proof of Theorem \ref{Thm:convexforhoroballs} is similar to the proof of the inclusion part in Theorem~\ref{Thm:convexforballs}. 

Let $b^s \colon M \to \mathbb R$ be a Busemann function associated to the ray $\gamma_s$. Recall that, by definition, $H(\gamma_s) = \{x \in M \colon b^s(x) \geqslant 0\}$ is a closed horoball centered at the ideal boundary point $\gamma_s(\infty)$, while the null set of $b^s(\cdot)$ defines the horosphere $\partial H(\gamma_s)$. Let $t$ be the arc length parameter for $\gamma_s$ with $\gamma_s(0) = s$.

In order to prove~(\ref{Eq:horoballincl}) it is enough to show that 
\begin{equation}
\label{Eq:horineq}
b^s(q) \geqslant 0 \quad \text{for every} \quad q \in \Sigma.
\end{equation}

Let us consider arbitrary $q \in \Sigma \sm \{s\}$, and as in the proof of Theorem~\ref{Thm:convexforballs}, pick a point $p \in D_\Sigma \sm \Sigma$ such that $\dist(p, \Sigma) = |pq| = \epsilon$ for some small $\epsilon$. 

Now we want to define the comparison objects. Let $\mathcal H$ be a horoball in the hyperbolic space $\mathbb H(-\lambda^2)$ defined with respect to a geodesic ray $\gamma_{s_\lambda} \colon [0, +\infty) \to \mathbb H(-\lambda^2)$, $\gamma_{s_\lambda}(0) = s_\lambda$ (we assume that $\gamma_{s_\lambda}$ is parameterized by the arc length). The boundary $\partial \mathcal H$ (a horosphere) of $\mathcal H$ is a totally umbilical hypersurface of curvature $\lambda$. Let $b^{s_\lambda}(\cdot)$ be the Busemann function corresponding to $\gamma_{s_\lambda}$. 

Pick $q_\lambda \in \partial \mathcal H$ and $p_\lambda \in \mathcal H$ such that 
$$
|s_\lambda p_\lambda| = |sp|, \quad |p_\lambda q_\lambda| = \dist(p_\lambda, \partial \mathcal H) = \epsilon.
$$ 

Suppose $\phi$ and $\phi_\lambda$ are the radial angle functions defined on $\Sigma$ and $\partial \mathcal H$ with respect to $p$ and $p_\lambda$, respectively. By the Radial Angle Comparison Theorem for $\lambda$-convex domains in Hadamard manifolds \cite{BGR01, BM02}, similar to (\ref{racbltheq0}), we must have
\begin{equation}
\label{Eq:horRAC}
\phi(s) \leqslant \phi_\lambda (s_\lambda).
\end{equation}

Pick an arbitrary $t>0$. Inequality (\ref{Eq:horRAC}) together with Toponogov's triangle comparison theorem, analogously to the proof of Theorem~\ref{Thm:convexforballs}, for the triangles $\triangle sp\gamma_s(t)$ and $\triangle s_\lambda p_\lambda \gamma_{s_\lambda}(t)$ yields
\begin{equation*}
\label{}
|p\gamma_s(t)| \leqslant |p_\lambda\gamma_{s_\lambda}(t)|.
\end{equation*}
Therefore, by the triangle inequality,
\begin{equation}
\label{Eq:horchain}
|q \gamma_s(t)| \leqslant |p \gamma_s(t)| + |pq| \leqslant |p_\lambda \gamma_{s_\lambda}(t)| + \epsilon = |p_\lambda \gamma_{s_\lambda}(t)| + |p_\lambda q_\lambda|,
\end{equation}
and thus
\begin{equation}
\label{Eq:horchain2}
t - |q \gamma_s(t)| \geqslant t - |p_\lambda \gamma_{s_\lambda}(t)| - |p_\lambda q_\lambda|.
\end{equation}
Sending $t \to \infty$ in the last inequality, we obtain
$$
b^s(q) \geqslant b^{s_\lambda}(p_\lambda) - |p_\lambda q_\lambda|
$$
At the same time,  
\begin{equation}
\label{Eq:hyphorprop}
 b^{s_\lambda}(p_\lambda) - |p_\lambda q_\lambda| = b^{s_\lambda}(q_\lambda) = 0
\end{equation}
since $\partial \mathcal H$ is horospheres in a hyperbolic space (namely, since $\partial \mathcal H$ is a horosphere and $q_\lambda p_\lambda \perp \partial \mathcal H$ , the geodesic ray $q_\lambda p_\lambda$ is passing through the same ideal boundary point as $\gamma_{s_\lambda}$, yielding the first equality in (\ref{Eq:hyphorprop}); the second equality in (\ref{Eq:hyphorprop}) follows since $q_\lambda \in \partial \mathcal H$). Therefore, $b^s(q) \geqslant 0$, which proves (\ref{Eq:horineq}).  
\qed

%%%%%%%%%%%%%%%%%%%%%%%%%%%%%%%%%%%%%%%%%%%%%%%%%%%%%%%%
%%%%%%%%%%%%%%%%%%%%%%%%%%%%%%%%%%%%%%%%%%%%%%%%%%%%%%%%

\section{Open questions}
\label{Sec:OpenQuestions}

In this section, we present a list of open questions related to the results of the paper.

\medskip

{\bf Question 1} (Asymptotic rigidity in Toponogov comparison). It is natural to expect a rigidity part in Theorem~\ref{Thm:convexforhoroballs} similar to part~\textbf{II} of Theorem~\ref{Thm:convexforballs}. As it follows from the proof in Section~\ref{Sec:ProofThmB}, to establish such a rigidity part it would require proving the following conjectural \emph{asymptotic rigidity} in Toponogov's comparison theorem: 

\smallskip

{\it
Let $M^m$ be an $m$-dimensional Hadamard manifold with $0>\sec(M^m) \ge - 1$. Let $b \colon M^m \to \mathbb R$ and $b_{1}\colon \mathbb H^m(-1) \to \mathbb R$ be a pair of Busemann functions defined with respect to geodesic rays $\gamma \colon [0, +\infty) \to M^m$ and $\gamma_{1} \colon [0, +\infty) \to \mathbb H^m(-1)$ starting at $x = \gamma(0)$ and $x_1 = \gamma_1(0)$ with the ideal boundary points $z^\infty = \gamma(+\infty)$ and $z_{1}^\infty = \gamma_{1}(+\infty)$. Suppose there exist points $y \in M^m$ and $y_1 \in \mathbb H^m(-1)$ such that
$$
|xy|=|x_1 y_1|, \quad b(y) = b_1(y_1), \quad \text{and}\quad \angle(\gamma, xy) = \angle(\gamma_1, x_1 y_1).
$$
Then the ideal triangle $x y z^\infty$ spans a totally geodesic submanifold isometric to the ideal triangle $x_1 y_1 z_1^\infty$ in $\mathbb H^2(-1)$.
}  

Note that in the statement above, $b(x) = b_1(x_1) = 0$ by construction.

\medskip

{\bf Question 2} (Equidistant rolling.) For negatively curved Riemannian manifolds, Theorems~\ref{Thm:convexforballs} and \ref{Thm:convexforhoroballs} establish a Riemannian version of Blaschke's rolling theorems in balls and horoballs. \textit{Is there a way to give a meaning to the rolling in convex domains bounded by `equidistant hypersurfaces' in the Riemannian sense? } 

\medskip
{\bf Question 3} (Inner rolling theorem in Riemannian setting). Dual to the notion of $\lambda$-convexity, there is a notion of \emph{$\lambda$-concavity}. A convex domain $D_\Sigma$ is called $\lambda$-concave if the normal curvatures of $\Sigma$ with respect to the inward pointing normals are bounded \emph{above} by $\lambda > 0$. 

\smallskip
\textit{Is there a version of Blaschke's rolling theorem for $\lambda$-concave domains in Riemannian manifolds with $\sec \le c$ under the sphere constraints \eqref{Eq:Sphere}?} 
\smallskip

In this setting, we expect that the balls roll inside the $\lambda$-concave domain. 

In general, the dual Blaschke's rolling theorem for the rolling of the ball inside a domain is not true even for constant curvature spaces, as the following construction shows. Take a 2-dimensional hyperbolic space in the unit disk model. Let $o$ be the center of the disk. Consider two disjoint hyperbolic balls of radii $R_\lambda$ symmetric with respect to $o$, and let $\mathcal C$ be the convex hull of these two balls. It is easy to see that the boundary $\partial \mathcal C$ of this hull will have be a convex curve of geodesic curvature at most $\lambda$ (except 4 points where it is not well defined). In fact, arcs with curvature $\lambda$ and $0$ will alternate. If the dual Blaschke's theorem was true, then a ball of radius $R_\lambda$ should roll freely inside $\mathcal C$. The last assertion is false, for instance, for the ball tangent at a closest to $o$ point on $\partial \mathcal C$. This example can be easily modified to make the curve smooth.

\smallskip
The example above partially motivates the following conjecture: {\it
Let $M$ be a smooth complete simply connected Riemannian manifold of dimension $m \ge 2$ with $\sec(M) \le c$. Suppose $\lambda\in \mathbb R$ satisfies the sphere constrains, and $R_\lambda$ is the corresponding radius of the sphere. Let $D_\Sigma \subset M$ be a $\lambda$-concave domain with smooth boundary compactly contained in some ball of radius $\inj(M)/2$. For $c < 0$, additionally assume that $D_\Sigma$ is $\sqrt{-c}$-convex. Then a ball of radius $R_\lambda$ can roll freely inside $D_\Sigma$. 
}

\end{document}